\documentclass[12pt,letterpaper]{article}

\usepackage{amssymb,amsfonts,amscd,amsthm}
\usepackage[all,arc]{xy}
\usepackage{enumerate, bbm}
\usepackage{mathrsfs}
\usepackage{tikz}
\usepackage[left=0.9in,top=0.9in,right=0.9in,bottom=0.9in]{geometry}
\usepackage{mathtools}
\usepackage{hyperref}
\usepackage{color}%colors
\usepackage{graphicx}
\usepackage{enumitem} %Allows resuming enumeration and numbering by section
\graphicspath{ {TexImages/} }

\newtheorem{thm}{Theorem}[section]
\newtheorem{cor}[thm]{Corollary}
\newtheorem{prop}[thm]{Proposition}
\newtheorem{lem}[thm]{Lemma}
\newtheorem{claim}[thm]{Claim}

\newtheorem{quest}[thm]{Question}

\newtheorem{conj}[thm]{Conjecture}
\newtheorem{obs}[thm]{Observation}

\theoremstyle{definition}

\hypersetup{
	colorlinks,
	citecolor=blue,
	filecolor=blue,
	linkcolor=blue,
	urlcolor=blue,
	linktocpage
}

\setenumerate[1]{label=\thesection.\arabic*.} %These two make your enumerated lists start with the section number
\setenumerate[2]{label*=\arabic*.} %Makes subparts labeled by numbers instead of letters.
\setlist[enumerate]{itemsep=2ex, topsep=2ex} %spaces out enumerate/itemize better
\setlist[itemize]{itemsep=2ex, topsep=2ex}

\newcommand{\E}{\mathbb{E}}

\newcommand{\al}{\alpha}
\newcommand{\be}{\beta}

\newcommand{\Gam}{\Gamma}
\newcommand{\sig}{\sigma}
\newcommand{\ep}{\epsilon}

\newcommand{\Om}{\Omega}

\newcommand{\del}{\delta}

\renewcommand{\l}{\left}
\renewcommand{\r}{\right}

\newcommand{\half}{\frac{1}{2}}

\newcommand{\sub}{\subseteq}

\renewcommand{\c}[1]{\mathcal{#1}}

\newcommand{\tr}[1]{\textrm{#1}}
\newcommand{\rec}[1]{\frac{1}{#1}}
\newcommand{\f}[2]{\frac{#1}{#2}}
\newcommand{\floor}[1]{\l\lfloor #1\r\rfloor}
\newcommand{\ceil}[1]{\l\lceil #1\r\rceil}

\newcommand{\tu}[1]{\textup{(#1)}}

\renewcommand{\S}{\mathfrak{S}}

\newcommand{\perm}{\mathfrak{S}_{m,n}}

\newcommand{\lm}{\mathcal{L}_{m}^{1}}

\renewcommand{\l}{\mathcal{P}_2}
\newcommand{\hm}{h_m(n)}

\newcommand{\wc}{\mathcal{N}(n,m)}

\newcommand{\roots}{\alpha_1,\ldots,\alpha_m}

\renewcommand{\l}{\left}

\title{Continuously Increasing Subsequences of Random Multiset Permutations}
\author{ Alexander Clifton\footnote{Dept.\ of Mathematics, Emory University {\tt aclift2@emory.edu}} \and Bishal Deb\footnote{Dept.\ of Mathematics,  University College London {\tt bishal.deb.19@ucl.ac.uk}} \and Yifeng Huang\footnote{Dept.\ of Mathematics, University of Michigan {\tt huangyf@umich.edu}} \and Sam Spiro\footnote{Dept.\ of Mathematics, UCSD {\tt sspiro@ucsd.edu}. This material is based upon work supported by the National Science Foundation Graduate Research Fellowship under Grant No. DGE-1650112.} \and Semin Yoo\footnote{School of Computational Sciences, KIAS {\tt syoo19@kias.re.kr}. The author is supported by the KIAS Individual Grant (CG082701) at
Korea Institute for Advanced Study.}}
\date{\today}
\begin{document}
	\maketitle
	\begin{abstract}
		For a word $\pi$ and integer $i$, we define $L^i(\pi)$ to be the length of the longest subsequence of the form $i(i+1)\cdots j$, and we let $L(\pi):=\max_i L^i(\pi)$.   In this paper we estimate the expected values of $L^1(\pi)$ and $L(\pi)$ when $\pi$ is chosen uniformly at random from all words which use each of the first $n$ integers exactly $m$ times.  We show that $\E[L^1(\pi)]\sim m$ if $n$ is sufficiently larger in terms of $m$ as $m$ tends towards infinity,  confirming a conjecture of Diaconis, Graham, He, and Spiro.  We also show that $\E[L(\pi)]$ is asymptotic to the inverse gamma function $\Gamma^{-1}(n)$ if $n$ is sufficiently large in terms of $m$ as $m$ tends towards infinity.
		
		%We say that a subsequence of a word $\pi$ is $i$-continuously increasing if it is of the form $i(i+1)\cdots j$ for some $j\ge i$, and we say that it is continuously increasing if it is $i$-continuously increasing for some $i$.  Let $L^i(\pi)$ denote the maximum length of an $i$-continuously increasing subsequence of $\pi$, and similarly define $L(\pi)$ for continuously increasing subsequences. 
	\end{abstract}

	\section{Introduction}
	\subsection{Main Results}
	
	Given integers $m$ and $n$, let $\S_{m,n}$ denote the set of words $\pi$ which use each integer in $[n]:=\{1,2,\ldots,n\}$ exactly $m$ times, and we will refer to $\pi\in \S_{m,n}$ as a \textit{multiset permutation}.  For example, $\pi=211323\in \S_{2,3}$.  For $\pi\in \S_{m,n}$ and $i$ an integer, we define $L^i_{m,n}(\pi)$ to be the length of the longest subsequence of $\pi$ of the form $i(i+1)(i+2)\cdots j$, which we call an \textit{$i$-continuously increasing subsequence}.  We say that a subsequence is a \textit{continuously increasing subsequence} if it is an $i$-continuously increasing subsequence for some $i$, and we define $L_{m,n}(\pi)=\max_i L^i_{m,n}(\pi)$ to be the length of a longest continuously increasing subsesquence of $\pi$.  For example, if $\pi=2341524315$ then $L_{2,5}(\pi)=L^2_{2,5}(\pi)=4$ due to the subsequence $2345$, and $L^1_{2,5}(\pi)=3$ due to the subsequence $123$.

	The focus of this paper is to study $L^1_{m,n}(\pi)$ and $L_{m,n}(\pi)$ when $\pi$  is chosen uniformly at random from $\S_{m,n}$.  We focus on the regime where $n$ is much larger than $m$, as in the opposite regime $L^i_{m,n}(\pi)$ is very likely to be its maximum possible value $n-i+1$ for all $i$.
	
	We first consider $\E[L^1_{m,n}(\pi)]$.  This quantity was briefly studied by Diaconis, Graham, He, and Spiro~\cite{DGHS21} due to its relationship with a certain card game that we describe later in this paper.  They proved $\E[L^1_{m,n}(\pi)]\le m+Cm^{3/4}\log m$ for some absolute constant $C$ provided $n$ is sufficiently large in terms of $m$.
	It was conjectured in \cite{DGHS21} that this upper bound for $\E[L^1_{m,n}(\pi)]$ is asymptotically tight for $n$ sufficiently large in terms of $m$.  We verify this conjecture in a strong form, obtaining an exact formula for $\lim_{n\to \infty} \E[L^1_{m,n}(\pi)]$ for any fixed $m$ and precise estimates of this value as $m$ tends towards infinity.
	
	\begin{thm}\label{thm:C1}
	~
	\begin{itemize}
	    \item[\tu{a}] For any integer $m\ge 1$, let $\al_1,\ldots,\al_m$ be the zeroes of $E_m(x):=\sum_{k=0}^m \f{x^k}{k!}$.  If $\pi\in \S_{m,n}$ is chosen uniformly at random, then
       \begin{equation}\label{expeq}
    \lm:=\lim_{n\to \infty} \E[L_{m,n}^1(\pi)]=-1-\sum \al_i^{-1} e^{-\al_i}.
    \end{equation}
    \item[\tu{b}] There exists an absolute constant $\be>0$ such that
		\[\l|\mathcal{L}_m^1-\l(m+1-\rec{m+2}\r)\r|\le O(e^{-\beta m}).\]
	\end{itemize}

	\end{thm}
	For example, when $m=1$ we have $E_1(x)=1+x$ and $\al_1=-1$, implying $\mathcal{L}_1^1=-1+e$, which can also be proven by elementary means.  For $m=2$ we have $E_2(x)=1+x+x^2/2$ and $\al_1=-1-i,\al_2=-1+i$.  From this Theorem~\ref{thm:C1}(a) gives the following closed form expression for $\c{L}_2^1$.
\begin{cor} 
\[
\mathcal{L}_{2}^{1}=e\left(\cos(1)+\sin(1)\right)-1.
\]
\end{cor}

	Our next result gives precise bounds for the length of a longest continuously increasing subsequence in a random permutation of $\S_{m,n}$.  We recall that the gamma function $\Gam(x)$ is a function which, in particular, gives a bijection from $x\ge 1$ to $y\ge 1$ and which satisfies $\Gam(n)=(n-1)!$ for non-negative integers $n$.  
	\begin{thm}\label{thm:C}
		If $n$ is sufficiently large in terms of $m$, then
		\[\E[L_{m,n}(\pi)]=\Gam^{-1}(n)+\Theta\l(1+\f{\log m}{\log(\Gam^{-1}(n))} \Gam^{-1}(n)\r),\]
		where $\Gam^{-1}(n)$ is the inverse of the gamma function when restricted to $x\ge 1$.
	\end{thm}
	Note when $m=1$ the error term of Theorem~\ref{thm:C} is $\Theta(1)$, but for $m\ge 2$ it is $\Theta(\f{\log m}{\log \Gam^{-1}(n)}\Gam^{-1}(n))$, which is fairly close to the main term of $\Gam^{-1}(n)$.  Thus the behavior of $\E[L_{m,n}(\pi)]$ changes somewhat dramatically as soon as one starts to consider multiset permutations as opposed to just permutations.  

	\subsection{History and Related Work}
	
	Determining $L^i_{m,n}(\pi)$ and $L_{m,n}(\pi)$ can be viewed as variants of the well-studied problem of determining the length of the longest increasing subsequence in a random permutation of length $n$, and we denote this quantity by $\widetilde{L}_n$.  It was shown by Logan and Shepp \cite{LS77} and Vershick and Kerov \cite{VK77} that $\E[\widetilde{L}_n]\sim 2\sqrt{n}$, answering a famous problem of Ulam. Later Baik, Deift, and Johansson \cite{BDJ99} showed that the limiting distribution of $\widetilde{L}_n$ is the Tracy-Widom distribution. Some work with the analogous problem for multiset permutations has been considered recently by Almeanazel and Johnson \cite{AMJ20}.  Much more can be said about this topic, and we refer the reader to the excellent book by Romik \cite{Rom} for more information.

	The initial motivation for studying $L^1(\pi)$ was due to its relationship to a card guessing experiment introduced by Diaconis and Graham~\cite{DG81}.  To start the experiment, one shuffles a deck of $mn$ cards which consists of $n$ distinct card types each appearing with multiplicity $m$. In each round, a subject iteratively guesses what the top card of the deck is according to some strategy $G$.  After each guess, the subject is told whether their guess was correct or not, the top card is discarded, and then the experiment continues with the next card.  This experiment is known as the \textit{partial feedback model}.  For more on the history of the partial feedback model we refer the reader to \cite{DGS21}.
	%The study of this and other card guessing feedback models originated from problems in clinical trials \cite{BH57, Efr71}, casino games \cite{Als12, EL05}, and ESP detection experiments \cite{Dia78}; see \cite{DGS21} for more information on feedback models.
	
	If $G$ is a strategy for the subject in the partial feedback model and $\pi\in \S_{m,n}$, we let $P(G,\pi)$ denote the number of correct guesses made by the subject if they follow strategy $G$ and the deck is shuffled according to $\pi$. We say that $G$ is an \textit{optimal strategy} if $\E[P(G,\pi)]=\max_{G'} \E[P(G',\pi)]$, where $G'$ ranges over all strategies and $\pi\in \S_{m,n}$ is chosen uniformly at random.   Optimal strategies are unknown in general, and even if they were known they would likely be too complex for a human subject to implement in practice.  As such there is interest in coming up with (simple) strategies $G$ such that $\E[P(G,\pi)]$ is relatively large.
	
	One strategy is the \textit{trivial strategy} which guesses card type 1 every single round, guaranteeing a score of exactly $m$ at the end of the experiment.  A slightly better strategy is the \textit{safe strategy} $G_{safe}$ which guesses card type 1 every round until all $m$ are guessed correctly, then 2's until all $m$ are guessed correctly, and so on.  It can be deduced from arguments given by Diaconis, Graham, and Spiro~\cite{DGS21} that $\E[P(G_{safe},\pi)]$ is $m+1-\rec{m+1}$ plus an exponential error term, so the safe strategy does just a little better than the trivial strategy.  
	
	Another natural strategy is the \textit{shifting strategy} $G_{shift}$, defined by guessing 1 until you are correct, then 2 until you are correct, and so on; with the strategy being defined arbitrarily in the (very rare) event that one correctly guesses a copy of each card type.  It is not difficult to see that $P(G_{shift},\pi)\ge L^1_{m,n}(\pi)$, with equality holding provided the player does not correctly guess $n$.  Thus Theorem~\ref{thm:C1}(b) shows that the expected number of correct guesses under the shifting strategy is close to $m+1-\rec{m+2}$, which is slightly better than the trivial strategy, and very slightly better than the safe strategy.
	
	\subsection{Preliminaries}
	We let $[n]:=\{1,2,\ldots,n\}$ and let $[m]^n$ be the set of tuples of length $n$ with entries in $[m]$.  Whenever we write, for example, $\Pr[L_{m,n}(\pi)\ge k]$, we will assume $\pi$ is chosen uniformly at random from $\S_{m,n}$ unless stated otherwise.
	
	Throughout this paper we use several basic results from probability theory.  One such result is that if $X$ is a non-negative integer-valued random variable, then
	\[\E[X]=\sum_{k=1}^{\infty}\Pr[X\ge k].\]
	A crucial observation that we use throughout the text is the following.
	\begin{obs}\label{obs1}
	    For $k\le n$, if $\pi\in \S_{m,n}$ and $\tau\in \S_{m,k}$ are drawn uniformly at random, then
	    \[\Pr[L^1_{m,n}(\pi)\ge k]=\Pr[L^1_{m,k}(\tau)=k].\]
	\end{obs}
	\begin{proof}
	    For $\pi\in \S_{m,n}$, let $\phi(\pi)\in \S_{m,k}$ be the word obtained by deleting every letter from $\pi$ which is larger than $k$.  Note that $L^1_{m,n}(\pi)\ge k$ if and only if $L^1_{m,k}(\phi(\pi))=k$.  Moreover, it is not difficult to see that $\phi(\pi)$ is distributed uniformly at random in $\S_{m,k}$ provided $\pi$ is distributed uniformly at random in $\S_{m,n}$, proving the result.
	\end{proof}
	
	\section{Proof of Theorem~\ref{thm:C1}}
	\subsection{Theorem \ref{thm:C1}(a): Generating Functions }

We say that a word $\pi\in \perm$ has a \textit{complete increasing subsequence} if $L^1_{m,n}(\pi) = n$. Let $\hm$ be the number of words $\pi\in\perm$ which have a complete increasing subsequence. Horton and Kurn \cite[Corollary (c)]{HK81} give the following formula for $\hm$. 
\begin{thm}[\cite{HK81}]\label{HK} The number of words $\pi\in\perm$ which have a complete increasing subsequence, $\hm$, is given by
\[
\hm = \sum_{(i_1,\ldots,i_m)\in\wc} \binom{n}{i_1,\ldots,i_m} \dfrac{(mn)!}{l!}\dfrac{(-1)^{l-n}}{\prod_{j=1}^m(m-j)!^{i_j}},
\]
where \[l=\sum_{j=1}^{m}j i_j,\] 
$\wc$ is the set of weak compositions of $n$ into $m$ parts, i.e.,
\[
\wc \coloneqq \left\lbrace (i_1,\ldots,i_m)\in \mathbb{Z}_{\geq{0}}^m  \middle|  \sum_{j=1}^m i_j = n\right\rbrace,\]
and \[\binom{n}{i_1,\ldots,i_m} = \dfrac{n!}{\prod_{j=1}^m i_j!}\] is a multinomial coefficient. 
\label{thm.hmformula}
\end{thm}
Notice that $\lm$ can be expressed in terms of $h_{m}(n)$ as follows: 
\begin{equation}
\mathcal{L}_{m}^{1}=\lim_{k\to \infty} \E[L_{m,k}^1(\pi)]= \lim_{k\to \infty} \sum_{n=1}^{k}\mathrm{Pr}[L^{1}_{m,k}(\pi)\geq n] = \lim_{k\to \infty} \sum_{n=1}^{k} \dfrac{\hm}{|\perm|}=\sum_{n=1}^{\infty} \dfrac{\hm}{|\perm|},
\label{eq.lmseries}
\end{equation}
where the third equality is due to Observation \ref{obs1}. Note that $|\perm| = (mn)!/(m!)^n$. Thus, as a consequence of Theorem~\ref{HK}, we have
\begin{subequations}
\begin{equation}
\dfrac{\hm}{|\perm|} = (-m!)^n \sum_{(i_1,\ldots,i_m)\in\wc} \binom{n}{i_1,\ldots,i_m} \dfrac{1}{\prod_{j=1}^m(m-j)!^{i_j}}\dfrac{(-1)^l}{l!},
\label{eq.summandexpand1}
\end{equation}
\begin{equation}
\hspace{2.5em}= (-m!)^n \sum_{(i_1,\ldots,i_m)\in\wc} \binom{n}{i_1,\ldots,i_m}\prod_{j=1}^m \left(\dfrac{(-1)^j}{(m-j)!}\right)^{i_j}\dfrac{1}{l!}.
\label{eq.summandexpand2}
\end{equation}
\end{subequations}

Intuitively, if the $1/l!$ were removed from the right-hand-side expression in \eqref{eq.summandexpand2}, then by using the multinomial theorem we could write this expression as an $n^{\text{th}}$ power, turning \eqref{eq.lmseries} into a geometric series. The next few paragraphs formalise this idea.

We begin by replacing $(-1)^l$ by $x^l$ in the right-hand-side of \eqref{eq.summandexpand1} to obtain the polynomial
\begin{align*}
    p_{m,n}(x)\coloneqq& (-m!)^n \sum_{(i_1,\ldots,i_m)\in\wc} \binom{n}{i_1,\ldots,i_m} \dfrac{1}{\prod_{j=1}^m(m-j)!^{i_j}}\dfrac{x^l}{l!}\\ 
    =& (-m!)^n \sum_{(i_1,\ldots,i_m)\in\wc} \binom{n}{i_1,\ldots,i_m}\prod_{j=1}^m \left(\dfrac{x^j}{(m-j)!}\right)^{i_j}\dfrac{1}{l!}.
\end{align*}

Thus,
\[
p_{m,n}(-1)= \dfrac{\hm}{|\perm|}.
\]

Next, we define an operator in order to remove the $l!$ from the denominator. Let $R$ be a commutative ring containing $\mathbb{Q}$ and let $\Phi:R[x]\to R[x]$ be an $R$-linear map defined on the monomials by 
\[
\Phi(x^n) = \dfrac{x^n}{n!}.
\]
We can extend $\Phi$ to an $R$-linear map on $R[[x]]\to R[[x]]$, which we also refer to as $\Phi$ by abuse of notation. Throughout this article, $R$ is either $\mathbb{C}$ or $\mathbb{C}[[y]]$ for an indeterminate $y$, and we shall refer to this $R$-linear map as $\Phi$ in both cases. Notice that $\Phi$ is invertible for any such ring $R$.  A key property that we use about $\Phi$ is
\begin{equation}\Phi\l(\rec{1-ax}\r)=\Phi\l(\sum_{i=0}^\infty (ax)^i\r)=\sum_{i=0}^\infty \f{(ax)^i}{i!}=e^{ax}.\label{eq:simplefraction}\end{equation}

Consider the polynomial
\begin{subequations}
\begin{equation*}
q_{m,n}(x)\coloneqq \Phi^{-1}\left(p_{m,n}(x)\right) = (-m!)^n \sum_{(i_1,\ldots,i_m)\in\wc} \binom{n}{i_1,\ldots,i_m} \dfrac{x^l}{\prod_{j=1}^m(m-j)!^{i_j}}
\end{equation*}
\begin{equation*}
= (-m!)^n \sum_{(i_1,\ldots,i_m)\in\wc} \binom{n}{i_1,\ldots,i_m}\prod_{j=1}^m \left(\dfrac{x^j}{(m-j)!}\right)^{i_j}.
\end{equation*}
\end{subequations}
Notice that,
\[
q_{m,n}(x) =  \left(-m!\sum_{j=1}^m \dfrac{x^j}{(m-j)!}\right)^n=(q_{m,1}(x))^n.
\]

Let $P_{m}(x,y)$ and $Q_{m}(x,y)$ be the ordinary generating functions of $p_{m,n}(x)$ and $q_{m,n}(x)$ respectively, i.e.
\begin{subequations}
\begin{equation*}
P_{m}(x,y) \coloneqq \sum_{n=0}^{\infty} p_{m,n}(x) y^n,
\end{equation*}
\begin{equation*}
Q_{m}(x,y) \coloneqq \sum_{n=0}^{\infty} q_{m,n}(x) y^n = \Phi^{-1}\left(P_{m}(x,y) \right).
\end{equation*}
\end{subequations}

Putting everything together, we see that
\begin{equation}\label{Cm1}
\lm=P_{m}(-1,1)-1.
\end{equation}
and thus it suffices to find a nice closed form expression for $P_{m}(x,y)$.  Note that
\[
q_{m,1}(x)= -m!x^m E_{m-1}(1/x),
\]
where we recall the polynomial $E_{m-1}(x)$ is defined in Theorem~\ref{thm:C1} by $E_{m-1}(x)=\sum_{k=0}^{m-1} x^k/k!$.   As $q_{m,n}(x) = \left(q_{m,1}(x)\right)^n $, we have
\begin{equation}\label{eq:Q}
Q_{m}(x,y) = \dfrac{1}{1-yq_{m,1}(x)} =\dfrac{1}{1+m!x^my E_{m-1}(1/x)}.
\end{equation}

Hence,
\[
P_{m}(x,y) =  \Phi\left(Q_{m}(x,y)\right) = \Phi\left(\dfrac{1}{1+m!x^my E_{m-1}(1/x)}\right),
\]
and thus
\[
P_{m}(x,1) = \Phi\left(\dfrac{1}{1+m!x^m E_{m-1}(1/x)}\right)=\Phi\left(\dfrac{1}{m!x^m E_{m}(1/x)}\right).
\]

We now prove the main result of this subsection.

\begin{prop}\label{zeroes} Let $\roots$ be the zeroes of the polynomial $E_m(x)$. The formal power series $P_{m}(x,1)$ satisfies
\begin{equation*}
P_{m}(x,1) = -\sum_{i=1}^m \alpha_i^{-1}e^{\alpha_{i}x }.
\end{equation*}
\end{prop}

\begin{proof} Let $g(x) \coloneqq m!x^m E_m(1/x)$. Since $\alpha_1^{-1},\ldots, \alpha_m^{-1}$ are the zeroes of $g(x)$, we have
\begin{equation*}
g(x) =m! (x-\alpha_1^{-1})\cdots (x-\alpha_m^{-1}).
\end{equation*}

Notice that $E_m(x)$ has no repeated zeroes. This is true because, if $\alpha$ is a repeated zero of $E_m(x)$, it is also a zero of its derivative $E_{m}'(x) = E_{m-1}(x)$. But then $\alpha$ has to be a zero of $E_m(x)-E_{m-1}(x) = x^m/m!$, which is only possible if $\alpha=0$, a contradiction as $0$ is not a zero of $E_m(x)$.

Thus  $\roots$ are pairwise distinct, and hence the zeroes of $g(x)$ being $\alpha_1^{-1},\ldots, \alpha_m^{-1}$, are also pairwise distinct. This and \eqref{eq:Q} means $Q_{m}(x,1)$ has the partial fraction decomposition
\begin{equation*}
Q_{m}(x,1) = \dfrac{1}{g(x)} = \sum_{i=1}^m \dfrac{1}{g'\left(\alpha_i^{-1}\right)}\cdot \dfrac{1}{x-\alpha_i^{-1}}.
\end{equation*}

The derivative of $g$ is
\[
g'(x) = m!\left( \dfrac{m x^m E_m(1/x)}{x} - \dfrac{x^mE_m'(1/x)}{x^2} \right)  = m!\left( \dfrac{m x^m E_m(1/x)}{x} - \dfrac{x^m(E_m(1/x) - x^{-m}/m!)}{x^2} \right)
\]

Hence for any $i$,
\begin{equation*}
g'\left(\alpha_i^{-1}\right) = \alpha_i^{2}
\end{equation*}
which gives
\begin{equation*}
P_{m}(x,1) = \Phi\left(Q_{m}(x,1)\right) = -\sum_{i=1}^m  \Phi\left(\dfrac{1}{\alpha_i(1-\alpha_i x)}\right) = -\sum_{i=1}^m \alpha_i^{-1}e^{\alpha_{i}x },
\end{equation*}
where this last step used \eqref{eq:simplefraction}.
\end{proof}
This proposition together with \eqref{Cm1} gives Theorem \ref{thm:C1}(a).

\subsection{Theorem \ref{thm:C1}(b): Exponential Sums of Zeroes}
We remind the reader that Theorem~\ref{thm:C1}(b) claims
\[\l|\mathcal{L}_m^1-\l(m+1-\rec{m+2}\r)\r|\le O(e^{-\beta m})\]
for some constant $\be>0$.  Given Theorem~\ref{thm:C1}(a), proving this claim is equivalent to showing that $\displaystyle\sum_i\alpha_i^{-1}e^{-\alpha_i}=-m-2+\frac{1}{m+2}+O(e^{-\beta m})$ for some positive constant $\beta$.  We do this by following the approach used by Conrey and Ghosh~\cite{conrey1988zeros} to evaluate a similar exponential sum.

For $1\le i\le m$, let $s_i$ and $t_i$ be the real and imaginary parts, respectively, of $\alpha_i$. Let $\gamma^-,\gamma,\gamma^+$ be arbitrary positive numbers such that $0<\gamma^-<\gamma<\gamma^+<1-\log{2}$. We partition $[m]$ into disjoint sets $S$ and $L$ where $i\in S$ when $s_i\leq{\gamma m}$ and $i\in L$ when $s_i>\gamma m$, allowing us to rewrite our desired sum as
\[\displaystyle\sum_{i=1}^m\alpha_i^{-1}e^{-\alpha_i}=\displaystyle\sum_{i\in S}\alpha_i^{-1}e^{-\alpha_i}+\displaystyle\sum_{i\in L}\alpha_i^{-1}e^{-\alpha_i}.\]

Define
\[R_m(x)=e^x-E_m(x)=\displaystyle\sum_{k=m+1}^\infty \frac{x^k}{k!}.\]
The following results are proven by Conrey and Ghosh \cite{conrey1988zeros} and Zemyan~\cite{Zem05}:
\begin{prop}\label{D1} 
~
\begin{itemize}

\item[\tu{a}] \textup{\cite[Equations (6) and (7)]{conrey1988zeros}} For sufficiently large $m$, we have
$|\alpha_i|\geq{me^{\gamma^--1}}$ for $i\in L$ and $|\alpha_i|\leq{me^{\gamma^+-1}}$ for $i\in S$. Consequently, we have $|\alpha_i|<m/2$ for $i\in S$.

\item[\tu{b}] \textup{\cite[Lemma 1]{conrey1988zeros}} For $|x|<\frac{1}{2}(m+2)$, we have
\[
\frac{1}{R_m(x)}=\frac{(m+1)!}{x^{m+1}}\left(1+\displaystyle\sum_{k=1}^\infty c_kx^k\right)
\]

with $|c_k|\leq{\frac{1}{2}(\frac{2}{m+2})^k}$.

\item[\tu{c}] \textup{\cite[Theorem 7]{Zem05}} \[\displaystyle\sum_{i=1}^m \alpha_i^{-t}=\left\{
                \begin{array}{ll}
               -1 & t=1, \\
               0 & 2\leq{t}\leq{m},\\
                  1/m! & t=m+1,\\
                  -1/m!&t=m+2
                  .
                \end{array}\right.
              \]
\end{itemize}
\end{prop}

We begin our argument by restricting our attention to the indices in $L$:
\begin{lem}\label{D2}
\[\l|\displaystyle\sum_{i\in L}\alpha_i^{-1}e^{-\alpha_i}\r|\leq{\gamma^{-1}e^{-\gamma m}}.\]
\end{lem}
\begin{proof}
By the triangle inequality, 
\[\l|\displaystyle\sum_{i\in L}\alpha_i^{-1}e^{-\alpha_i}\r|\leq{\displaystyle\sum_{i\in L}|\alpha_i^{-1}e^{-\alpha_i}|}.\]
Note that $|\alpha_i^{-1}e^{-\alpha_i}|=|\alpha_i^{-1}||e^{-s_i}|<|\alpha_i^{-1}|e^{-\gamma m}$. Since $s_i>\gamma m$, we know that $|\alpha_i|>\gamma m$, so $|\alpha_i^{-1}|<(\gamma m)^{-1}$. Thus for $i\in L$, we have that 
\[
|\alpha_i^{-1}e^{-\alpha_i}|<(\gamma m)^{-1}e^{-\gamma m}.
\]
Adding over the elements of $L$, of which there are at most $m$, we have $|\displaystyle\sum_{i\in L}\alpha_i^{-1}e^{-\alpha_i}|\leq{\gamma^{-1}e^{-\gamma m}}$.
\end{proof}

To evaluate the sum for the indices in $S$, we utilize $R_m(x)$.  Since the $\alpha_i$'s are the roots of $E_m(x)$, we have $e^{\alpha_i}=R_m(\alpha_i)$ for $i=1,\cdots,m$, so $e^{-\alpha_i}=\frac{1}{R_m(\alpha_i)}$.

\begin{lem}\label{D3}
For $|x|<\frac{1}{2}(m+2)$, we have
\[
\frac{1}{R_m(x)}=\frac{(m+1)!}{x^{m+1}}\left(1-\frac{x}{m+2}+\displaystyle\sum_{k=2}^\infty c_kx^k\right)
\]

with $|c_k|\leq{\frac{1}{2}(\frac{2}{m+2})^k}$.
\end{lem}
\begin{proof}

All of this follows from Proposition \ref{D1}(b) except for showing that $c_1=-1/(m+2)$.  To show this, we observe by definition of $R_m(x)$ that

\[
\frac{x^{m+1}}{(m+1)!R_m(x)}=\left(1+(m+1)!\displaystyle\sum_{k=1}^{\infty}\frac{x^k}{(m+1+k)!}\right)^{-1}.
\]

Conrey and Ghosh~\cite{conrey1988zeros} note that $\left|(m+1)!\displaystyle\sum_{k=1}^{\infty}\frac{x^k}{(m+1+k)!}\right|<1$ when $|x|<\frac{1}{2}(m+2)$, so this can be expanded as a convergent geometric series in such cases.
Thus, \[
\frac{1}{R_m(x)}=\frac{(m+1)!}{x^{m+1}}\left(1+\displaystyle\sum_{j=1}^\infty \left(-(m+1)!\displaystyle\sum_{k=1}^{\infty}\frac{x^k}{(m+1+k)!}\right)^j\right).
\]

The only time an $x$ term can appear in the double infinite sum is when $k,j=1$, so this term has coefficient $-(m+1)!\frac{1}{(m+1+1)!}=\frac{-1}{m+2}$ as desired.

\end{proof}

Note that for $i\in S$ we have $|\alpha_i|<m/2$ from Proposition \ref{D1}(a). Thus we can use Lemma \ref{D3} to conclude that \begin{equation}\displaystyle\sum_{i\in S}\alpha_i^{-1}e^{-\alpha_i}=\displaystyle\sum_{i\in S}\frac{\alpha_i^{-1}}{R_m(\alpha_i)}=(m+1)!\displaystyle\sum_{i\in S}\displaystyle\sum_{k=0}^{\infty}c_k\alpha_i^{k-m-2},\label{eq.Ssum}\end{equation}
where $c_0=1$ and $c_1=\frac{-1}{m+2}$.

Using the values of $\displaystyle\sum_{i=1}^m \alpha_i^{-t}$ for $t=1,\cdots,m+2$ from Proposition~\ref{D1}(c) and that $c_0=1$ and $c_1=\f{-1}{m+2}$, we see that

{\footnotesize \[
\displaystyle\sum_{k=0}^{m+1}\displaystyle\sum_{i\in S}c_k\alpha_i^{k-m-2}=\displaystyle\sum_{k=0}^{m+1}\l(\displaystyle\sum_{i=1}^m c_k\alpha_i^{k-m-2}-\displaystyle\sum_{i\in L}c_k\alpha_i^{k-m-2}\r)=\frac{-1}{m!}-\frac{1}{m!(m+2)}-c_{m+1}-\displaystyle\sum_{k=0}^{m+1}\displaystyle\sum_{i\in L}c_k\alpha_i^{k-m-2}.
\]}
Plugging this into \eqref{eq.Ssum} gives
{\small \begin{align*}
\displaystyle\sum_{i\in S}\alpha_i^{-1}e^{-\alpha_i}&=(m+1)!\left(\frac{-1}{m!}-\frac{1}{m!(m+2)}-c_{m+1}-\displaystyle\sum_{k=0}^{m+1}\displaystyle\sum_{i\in L}c_k\alpha_i^{k-m-2}+\displaystyle\sum_{i\in S}\displaystyle\sum_{k=m+2}^{\infty}c_k\alpha_i^{k-m-2}\right)\\
&=-(m+1)-\frac{m+1}{m+2}-(m+1)!\left(c_{m+1}+\displaystyle\sum_{k=0}^{m+1}\displaystyle\sum_{i\in L}c_k\alpha_i^{k-m-2}-\displaystyle\sum_{i\in S}\displaystyle\sum_{k=m+2}^{\infty}c_k\alpha_i^{k-m-2}\right)\\
&=-m-2+\frac{1}{m+2}-(m+1)!\left(c_{m+1}+\displaystyle\sum_{k=0}^{m+1}\displaystyle\sum_{i\in L}c_k\alpha_i^{k-m-2}-\displaystyle\sum_{i\in S}\displaystyle\sum_{k=m+2}^{\infty}c_k\alpha_i^{k-m-2}\right).
\end{align*}}

By Lemma \ref{D2}, we have $\displaystyle\sum_{i=1}^m\alpha_i^{-1}e^{-\alpha_i}=\displaystyle\sum_{i\in S}\alpha_i^{-1}e^{-\alpha_i}+O(e^{-\beta m})$ for some $\beta>0$, so we are left to consider the sum over $S$. The first three terms above match our claimed expression, so it suffices to show that the leftover terms $-(m+1)!\left(c_{m+1}+\displaystyle\sum_{k=0}^{m+1}\displaystyle\sum_{i\in L}c_k\alpha_i^{k-m-2}-\displaystyle\sum_{i\in S}\displaystyle\sum_{k=m+2}^{\infty}c_k\alpha_i^{k-m-2}\right)$ are $O(e^{-\beta m})$.

%To do this, we first fix positive numbers $\gamma^-$ and $\gamma^+$ such that $\gamma^-<\gamma<\gamma^+<1-\log{2}$. Then, (\color{red}{WHY}\color{black}), for $i\in L$, we have

Using the Triangle Inequality, and recalling that $|c_k|\leq{(\frac{2}{m})^k}$ for all $k$ by Proposition~\ref{D1}(b), we obtain

\begin{align*}
&\left|c_{m+1}+\displaystyle\sum_{k=0}^{m+1}\displaystyle\sum_{i\in L}c_k\alpha_i^{k-m-2}-\displaystyle\sum_{i\in S}\displaystyle\sum_{k=m+2}^{\infty}c_k\alpha_i^{k-m-2}\right|\\\leq&{(2/m)^{m+1}+\displaystyle\sum_{k=0}^{m+1}\displaystyle\sum_{i\in L}(2/m)^k|\alpha_i|^{k-m-2}+\displaystyle\sum_{i\in S}\displaystyle\sum_{k=m+2}^{\infty}(2/m)^k|\alpha_i|^{k-m-2}}.
\end{align*}

Since $|L|,|S|\leq{m}$, this is at most
$(2/m)^{m+1}+m\displaystyle\sum_{k=0}^{m+1}(2/m)^k|\alpha_i|^{k-m-2}+m\displaystyle\sum_{k=m+2}^{\infty}(2/m)^k|\alpha_i|^{k-m-2}$.

Now we make use of Proposition \ref{D1}(a). In the first summation, the quantity $k-m-2$ is negative, so $|\alpha_i|^{k-m-2}\leq{(me^{\gamma^--1})^{k-m-2}}$.  In the second summation, $k-m-2$ is nonnegative, so $|\alpha_i|^{k-m-2}\leq{(me^{\gamma^+-1})^{k-m-2}}$. Putting this altogether, we have
\begin{align*}
&\left|c_{m+1}+\displaystyle\sum_{k=0}^{m+1}\displaystyle\sum_{i\in L}c_k\alpha_i^{k-m-2}-\displaystyle\sum_{i\in S}\displaystyle\sum_{k=m+2}^{\infty}c_k\alpha_i^{k-m-2}\right|\\&\leq{(2/m)^{m+1}+m\displaystyle\sum_{k=0}^{m+1}(2/m)^k(me^{\gamma^--1})^{k-m-2}+m\displaystyle\sum_{k=m+2}^{\infty}(2/m)^k(me^{\gamma^+-1})^{k-m-2}}\\
&=(2/m)^{m+1}+\frac{m^{-m-1}}{(e^{\gamma^--1})^{m+2}}\displaystyle\sum_{k=0}^{m+1}(2e^{\gamma^--1})^k+\frac{m^{-m-1}}{(e^{\gamma^+-1})^{m+2}}\displaystyle\sum_{k=m+2}^{\infty}(2e^{\gamma^+-1})^k.
\end{align*}

Note that the first summation is finite and it is bounded above by the convergent infinite sum $\sum_{k=0}^\infty (2e^{\gamma^--1})^k$, which is a constant. Since $0<\gamma^+<1-\log{2}$, we have that $2e^{\gamma^+-1}<1$, so the second summation also converges. Let $C$ be some constant which serves as an upper bound for both of these summations.  The total expression is then at most
\[
(2/m)^{m+1}+2C\frac{m^{-m-1}}{(e^{\gamma^--1})^{m+2}}.
\]

We need to show that this expression will still be $O(e^{-\beta m})$ after we multiply it by $(m+1)!$. By the Stirling approximation, $m!$ is asymptotically $\sqrt{2\pi m}(\frac{m}{e})^{m}$, so for sufficiently large $m$, we have  
\[
(m+1)!=(m+1)m!\leq{(m+1)(m-1)(m/e)^{m}}\leq{m^2(m/e)^{m}}.
\]

Now we examine \begin{align*}m^2(m/e)^{m}\left((2/m)^{m+1}+2C\frac{m^{-m-1}}{(e^{\gamma^--1})^{m+2}}\right)&=\frac{2^{m+1}}{e^m}m+2C\frac{m}{e^m(e^{\gamma^--1})^{m+2}}.\end{align*}

Let $D:=(e^{\gamma^--1})^{-1}$. Since $\gamma^-<1-\log{2}$, we have that $e^{\gamma^--1}<\frac{1}{2}$ and hence that $D>2$. However, $\gamma^-$ can be chosen arbitrarily close to $1-\log{2}$ to ensure $D<e$.  In this case

\begin{align*}m^2(m/e)^{m}\left((2/m)^{m+1}+2C\frac{m^{-m-1}}{(e^{\gamma^--1})^{m+2}}\right)&=me^{-m}\left(2^{m+1}+2CD^{m+2}\right)\\
&=O\left(me^{-m}(2^m+CD^2D^m)\right)\\
&=O\left(me^{-m}D^m\right)\\
&=O\left(e^{m(\frac{\log{m}}{m}-1+\log{D})}\right).
\end{align*}
Note that $-1+\log{D}<0$ and that the $\frac{\log{m}}{m}$ is negligible for sufficiently large $m$, so indeed the sum we are considering is $O(e^{-\beta m})$ for some positive $\beta$, completing our proof of Theorem~\ref{thm:C1}(b).
	
	\section{Proof of Theorem~\ref{thm:C}}
	At a high level, the proof of Theorem~\ref{thm:C}  revolves around showing that $\Pr[L_{m,n}(\pi)\ge k]$ tends to 0 or 1 depending on if $nm^k/k!$ tends to 0 or infinity as $n$ tends to infinity.  The following lemma (which we will apply for $t!\approx n$) will be used to determine the threshold when $nm^k/k!$ shifts from being very small to very large.  Here and throughout the text, $\log$ denotes the natural logarithm.
	\begin{lem}\label{lem:asy}
	Given integers $m\ge 1$ and $t\ge 2$, let $C>0$ be a real number such that $k=t+\f{C\log m}{\log t}t$ is an integer. We have
	
	\[k!\ge t! \cdot m^{Ct},\]
	and if $t\ge m^{10C}$, we have
	\[k!\le t!\cdot (1.1)^k m^{Ct}\]
	\end{lem}
	Here and throughout the text we define the falling factorial $(N)_a:=N(N-1)\cdots(N-a+1)$. 
	\begin{proof}
		Note that $k!=t!(k)_{k-t}$, so it suffices to show \[m^{Ct}\le (k)_{k-t}\le (1.1)^k m^{Ct},\] with the upper bound holding when $t\ge m^{10C}$. When $t\ge m^{10C}$, we have $\log t\ge 10C\log m$, so $k=t+\f{C\log m}{\log t}t \le 1.1t$.  This implies
		\[(k)_{k-t}\le k^{k-t}\le (1.1t)^{k-t}\le (1.1)^k \cdot t^{k-t}=(1.1)^k\cdot m^{Ct}.\]
		Similarly $(k)_{k-t}\ge t^{k-t}=m^{Ct}$ for all $t$, proving the result.
	\end{proof}
	
	Before delving into the details of the proof, we introduce some auxiliary definitions that will make our arguments somewhat cleaner. The main idea is that we wish to reduce multiset permutations to set permutations by labeling each of the $m$ copies of $i\in [n]$.  
	
	To this end, let $\S_{m,n}^*$ denote the set of permutations of the set $\{i_h: i\in [n],\ h\in [m]\}$.  For example, $\tau':=3_12_13_21_22_21_1\in \S_{2,3}^*$.  If $\tau\in \S_{m,n}^*$ contains a subsequence of the form $(w_1)_{x_1}\cdots (w_k)_{x_k}$, then we will say that $\tau$ has a subsequence of type $(w,x)$ where $w=w_1\cdots w_k$ and $x=x_1\cdots x_k$.  We say that $\tau\in \S_{m,n}^*$ has a subsequence of type $w$ if it has a subsequence of type $(w,x)$ for some $x$.  For example, $\tau'$ defined above has a subsequence of type $(12,22)$ and hence of type $12$, but it contains no subsequence of type $123$.
	\begin{obs}\label{obs2}
		If $\pi\in \S_{m,n}$ and $\tau\in \S_{m,n}^*$ are chosen uniformly at random, then for any word $w$ with letters in $[n]$ we have
		\[\Pr[\pi\tr{ contains }w\tr{ as a subsequence}]=\Pr[\tau\tr{ contains a subsequence of type }w].\]
	\end{obs}
	The intuition for this observation is as follows.  We can view $\{i_h: i\in [n],\ h\in [m]\}$ as a deck of cards with $n$ card types each having $m$ suits, and we view $\tau\in \S_{m,n}^*$ as a way of shuffling this deck.  The property that $\tau$ contains a subsequence of type $w$ is independent of the suits of the cards.  Thus if we let $\pi\in \S_{m,n}$ denote the shuffling $\tau$ after ignoring suits, then $\pi$ contains $w$ as a subsequence if and only if $\tau$ contains a subsequence of type $w$.  More formally, one can prove this result by considering the map $\phi:\S_{m,n}^*\to \S_{m,n}$ which deletes the subscripts in the letters of $\tau\in \S_{m,n}^*$.  We omit the details.

	\subsection{The Upper Bound}
	To prove the upper bound of Theorem~\ref{thm:C}, essentially the only fact we need is that there are at most $n$ continuously increasing subsequences of a given length $k$, and as such our proof easily generalizes to a wider set of subsequence problems.
	
	To this end, let $\c{W}$ be a set of words with letters in $[n]$.  For $\pi\in \S_{m,n}$, we define $L_{m,n}(\pi;\c{W})$ to be the maximum length of a word $w\in \c{W}$ which appears as a subsequence in $\pi$.   For example, if $\c{W}$ consists of every word of the form $i(i+1)\cdots j$ for some $i\le j$, then $L_{m,n}(\pi;\c{W})=L_{m,n}(\pi)$.  We will say that a set of words $\c{W}$ is \textit{prefix closed} if for every $w_1\cdots w_k\in \c{W}$ we have $w_1\cdots w_\ell \in \c{W}$ for all $\ell\le k$.

	\begin{lem}\label{lem:contUpper}
		Let $\c{W}$ be a prefix closed set of words with letters in $[n]$ and let $\c{W}_k\sub \c{W}$ be the set of words of length $k$ in $\c{W}$.  If $\pi\in \S_{m,n}$ is chosen uniformly at random, then
		\[\Pr[L_{m,n}(\pi;\c{W})\ge k]\le \f{|\c{W}_k|m^k}{k!}.\]
	\end{lem}
	\begin{proof}
		For $\tau\in \S_{m,n}^*$ we define $L_{m,n}^*(\tau;\c{W})$ to be the length of a longest $w\in \c{W}$ such that $\tau$ contains a subsequence of type $w$.  By Observation~\ref{obs2}, it suffices to bound  $\Pr[L_{m,n}^*(\tau;\c{W})\ge k]$ with $\tau$ chosen uniformly at random from $\S_{m,n}^*$.
		
		Because $\c{W}$ is prefix closed, we have $L_{m,n}^*(\tau;\c{W})\ge k$ if and only if $\tau$ contains some subsequence of type $w\in \c{W}_k$, and by definition this happens if and only if $\tau$ contains some subsequence of type $(w,x)$ with $w\in \c{W}_k$ and $x\in [m]^k$.  For $w\in \c{W}_k$ and $x\in [m]^k$, let $1_{w,x}(\tau)$ be the indicator variable which is 1 if $\tau$ contains a subsequence of type $(w,x)$ and which is 0 otherwise.  Let $X(\tau)=\sum_{w\in \c{W}_k,x\in [m]^k} 1_{w,x}(\tau)$.  By our observations above and Markov's inequality, we find
		\[\Pr[L_{m,n}^*(\tau;\c{W})\ge k]=\Pr[X(\tau)\ge 1]\le \E[X(\tau)]=\sum_{w\in \c{W}_k,x\in [m]^k} \Pr[1_{w,x}(\tau)=1]=|\c{W}_k|m^k\cdot \rec{k!},\]
		where the last step used that $1_{w,x}(\tau)=1$ if and only if the distinct letters $(w_1)_{x_1},\ldots,(w_k)_{x_k}$ appear in the correct relative order in $\tau$, and this happens with probability $1/k!$.  This proves the result.
	\end{proof}
	\begin{prop}\label{prop:prefix}
		Let $\c{W}$ be a prefix closed set of words with letters in $[n]$ and let $\c{W}_k\sub \c{W}$ be the words of length $k$ in $\c{W}$.  Assume there exists an $N$ such that $|\c{W}_k|\le N$ for all $k$. If $\pi\in \S_{m,n}$ is chosen uniformly at random and $N$ is sufficiently large in terms of $m$, then
		\[\E[L_{m,n}(\pi;\c{W})]\le \Gam^{-1}(N)+O\l(1+\f{\log m}{\log (\Gam^{-1}(N))} \Gam^{-1}(N)\r).\]	
	\end{prop}
	In particular, for $\c{W}$ the set of continuously increasing words $i(i+1)\cdots j$, we have $|\c{W}_k|\le n$ for all $k$, so taking $N=n$ gives the upper bound of Theorem~\ref{thm:C}.  As another example, if $\c{W}$ is the set of arithmetic progressions, then one can take $N=n^2$ to give an upper bound of roughly $\Gamma^{-1}(n^2)$ for $\E[L_{m,n}(\pi;\c{W})]$. Recent work of Goh and Zhao~\cite{goh2020arithmetic} shows that this bound for arithmetic progressions is tight.
	\begin{proof}
		 By using Lemma~\ref{lem:contUpper} and the trivial bound $\Pr[L_{m,n}(\pi;\c{W})\ge K]\le 1$, we find for all integers $k$ that
		\begin{align}\E[L_{m,n}(\pi;\c{W})]&=\sum_{K\ge 1} \Pr[L_{m,n}(\pi;\c{W})\ge K]\le k+\sum_{K>k} \f{N m^K}{K!}.\label{eq:contUpper}\end{align}

		Let $t$ be the integer such that $(t-1)!<N\le t!$ and let $k=\ceil{t+\f{2\log m}{\log t}t}$.  Note that this implies $k=t+\f{C \log m}{\log t}t$ for some $C\ge 2$.  Assume $N$ is sufficiently large so that $2m\le k\le 2t$.  By Lemma~\ref{lem:asy}, we have for $K>k\ge t$ that
		\[\f{N m^K}{K!}\le \f{Nm^k}{k!}\cdot \l(\f{m}{k}\r)^{K-k}\le \f{Nm^k}{t!m^{Ct}}\cdot 2^{k-K}\le 2^{k-K},\]
		with this last step using $k\le 2t\le Ct$ and $N\le t!$.	Plugging this and our choice of $k$ into \eqref{eq:contUpper} gives, after setting $\ell=K-k$,
		\[\E[L_{m,n}(\pi;\c{W})]\le \ceil{t+\f{2\log m}{\log t}t}+\sum_{\ell >0} 2^{-\ell}\le t+\f{2\log m}{\log t}t+2.\]
		This gives the desired result since $t< \Gam^{-1}(N)$.
	\end{proof}
	%\SS{If one instead takes $k=\ceil{t+(1+2/t)(\log m/\log t)t}$ then this same proof should give $t+(\log m/\log t)t+O(1+\log m/\log t)$ which is asymptotically tight.}
	
	%\SS{With regards to the dependencies: we really need $t\ge m^2$ for this to work, which roughly means we need $n>(m^2)!$ for this lower bound to hold.}
	
	%\SS{A Somewhat interesting case that isn't quite covered by this is when $\c{W}$ consists of all cyclic shifts of words of the form $i(i+1)\cdots j$.  Literally the proposition does apply here with $N=n^2$, but the ``real'' bound should be $N=kn$ which would give a tight bound.}
	
	\subsection{The Lower Bound}
	For $x,y\in [m]^n$, we define their \textit{Hamming distance} $d_H(x,y):=|\{i\in[n]: x_i\ne y_i\}|$.  Our key lemma for proving the lower bound of Theorem~\ref{thm:C} is the following:
	\begin{lem}\label{lem:inclusionExclusion}
		Let $T\sub [m]^n$ be such that any distinct $x,y\in T$ have $d_H(x,y)\ge \del$ for some integer $\del$.  Then
		\[\Pr[L_{m,n}(\pi)=n]\ge \f{|T|}{n!}\l(1-\f{|T|}{\del!}\r).\]
	\end{lem}
	
	\begin{proof}
		For $\tau\in \S_{m,n}^*$, let  $L^*_{m,n}(\tau)$ denote the length of the longest subsequence of $\tau$ of type $i(i+1)\cdots j$. By Observation~\ref{obs2}, it suffices to prove this lower bound for $\Pr[L_{m,n}^*(\tau)=n]$ where $\tau\in \S_{m,n}^*$ is chosen uniformly at random.  For $x\in [m]^n$, let $A_x(\tau)$ be the event that $\tau$ contains a subsequence of type $(12\cdots n,x)$.  Observe that
		 \begin{align}\Pr[L_{m,n}^*(\tau)=n]&=\Pr[\bigcup_{x\in [m]^n}A_x(\tau)]\ge \Pr[\bigcup_{x\in T}A_x(\tau)]\nonumber\\&\ge \sum_{x\in T}\Pr[A_x(\tau)]-\sum_{x,y\in T,\ x\ne y}\Pr[A_x(\tau)\cap A_y(\tau)],\label{eq:contLower}\end{align}
		where the last inequality used the Bonferroni inequality (which is essentially a weakening of the principle of inclusion-exclusion); see e.g. \cite{spencer2014asymptopia} for further details on this inequality.  To bound \eqref{eq:contLower}, we use the following: \begin{claim}
		If $x,y\in T$ with $x\ne y$, then $\Pr[A_x(\tau)]=1/n!$ and \[\Pr[A_x(\tau)\cap A_y(\tau)]\le \f{1}{\del!n!}.\]
		\end{claim}
		\begin{proof}
		    Observe that $A_x(\tau)$ occurs if and only if $1_{x_1},\ldots,n_{x_n}$ occur in the correct relative order in $\tau$, so $\Pr[A_x(\tau)]=1/n!$.  Let $S=\{i_1<i_2<\cdots<i_\del\}$ be any set of $\del$ indices $i$ such that $y_i\ne x_i$, and note that such a set exists by assumption of $T$.  Let $A_y^S(\tau)$ be the event that $(i_1)_{y_{i_1}}\cdots (i_\del)_{y_{i_\del}}$ is a subsequence of $\tau$.  Observe that $\Pr[A_y^S(\tau)]=1/\del!$ and that this event is independent of $A_x(\tau)$ since these two events concern disjoint sets of letters.  Because $A_y(\tau)$ implies $A_y^S(\tau)$, we have
		    \[\Pr[A_x(\tau)\cap A_y(\tau)]\le \Pr[A_x(\tau)\cap A_y^S(\tau)]=\rec{\del! n!},\]
		    proving the result.
		\end{proof}
		Plugging the results of this claim into \eqref{eq:contLower} and using that the second sum of  \eqref{eq:contLower} has at most $|T|^2$ terms gives the desired result.
	\end{proof}
	
	The problem of finding $T\sub [m]^n$ such that $d_H(x,y)\ge \del$ with $|T|$ and $\del$ both large is the central problem of coding theory.   In particular, a basic greedy argument from coding theory gives the following:
	\begin{lem}\label{lem:badPairs}
		For any $m\ge 2$ and $1\le \del\le n/2$, there exists $T\sub [m]^n$ such that any two distinct $x,y\in T$ have $d_H(x,y)\ge \del$  and such that \[|T|\ge \f{m^n}{\del {n\choose \del} (m-1)^\del}.\]
	\end{lem}
	\begin{proof}
	    Let $T\sub [m]^n$ be a set such that $d_H(x,y)\ge \del$ for distinct $x,y\in T$ and such that $|T|$ is as large as possible.  Let $B(x)=\{y\in [m]^n:d_H(x,y)<\del\}$, and note that for all $x$, \[|B(x)|= \sum_{d=0}^{\del -1} {n\choose d}(m-1)^d\le \del {n\choose \del}(m-1)^\del,\]
	    with this last step using ${n\choose d}\le {n\choose \del}$ for $d<\del\le n/2$.   By the maximality of $|T|$, we must have $[m]^n=\bigcup_{x\in T} B(x)$, and thus
	    \[m^n=|\bigcup_{x\in T} B(x)|\le |T|\cdot \del {n\choose \del}(m-1)^\del,\]
	    giving the desired bound on $|T|$.
	\end{proof}
	Combining  Lemmas~\ref{lem:inclusionExclusion} and \ref{lem:badPairs} gives the following:
	\begin{prop}\label{prop:lowerCont}
		For $n$ sufficiently large in terms of $m\ge 2$, we have
		\[\Pr[L_{m,n}(\pi)=n]\ge \f{(m/1.03)^n}{2n\cdot n!}.\]
	\end{prop}
	\begin{proof}
		We start with the following fact.
		\begin{claim}
		For all $\ep>0$, there exists a constant $0<c_\ep\le 1$ such that if $\del\le c_\ep n$, then ${n\choose \del}\le (1+\ep)^n$.
		\end{claim}
		\begin{proof}
		    In \cite{cover1999elements} it is noted that ${n\choose \del}\le 2^{H(\del/n) n}$ for all $n,\del$, where $H(p):=-p\log_2(p)-(1-p)\log_2(1-p)$ is the binary entropy function.  Because $H(p)$ tends to 0 as $p$ tends to 0, there exists a constant $c$ such that $2^{H(c)}\le 1+\ep$, and the result follows by taking $c_\ep=c$.
		\end{proof}
		Let $\del=\f{2\log m}{\log n} n$, and assume $n$ is sufficiently large in terms of $m$ so that $\del\le c_{.01}n$, i.e. so that ${n\choose \del}\le (1.01)^n$.  We also choose $n$ sufficiently large so that $\del\le \f{\log 1.01}{\log m}n$, or equivalently so that $m^\del\le (1.01)^n$. Let $T$ be a set as in Lemma~\ref{lem:badPairs}, and by our assumptions above we find
		\[|T|\ge \f{(m/1.03)^n}{n}.\]
		Possibly by deleting elements from $T$ we can assume that $|T|$ is exactly the quantity stated above\footnote{Strictly speaking we should take $|T|$ to be the floor of this value to guarantee that it is an integer.  This would change our ultimate bound by at most a factor of 2, and this factor of 2 can easily be recovered by sharpening our analysis in various places.}, so by Lemma~\ref{lem:inclusionExclusion} it suffices to show $|T|/\del!\le \half$.  Using the inequality $\del!\ge (\del/e)^\del$ and that $n$ is sufficiently large, we have 
		\[\del!\ge \exp\l[\del\cdot (\log(\del)-1)\r]\ge \exp\l[\f{2\log m}{\log n}n\cdot (\log(n)-\log(\log(n))-1)\r]\ge \exp\l[\log(m) n\r]=m^n.\]
		Thus $|T|/\del!\le (1.03)^{-n}/n\le \half$, proving the result.
	\end{proof}
	%\SS{With regards to dependencies: this proof roughly works for $n\ge m^{\log m}$.}
	
	With this we can now prove Theorem~\ref{thm:C}.
	\begin{proof}[Proof of Theorem~\ref{thm:C}]
		The upper bound follows from Proposition~\ref{prop:prefix}.  To prove the lower bound, fix an integer $k$.  For $0\le j< \floor{n/k}$, let $A_j(\pi)$ be the event that $\pi$ contains the subsequence $(jk+1)(jk+2)\cdots((j+1)k)$.  \begin{claim}
		We have the following:
		\begin{itemize}
			\item[\tu{a}] If any $A_j(\pi)$ event occurs, then $L_{m,n}(\pi)\ge k$,
			\item[\tu{b}] The $A_j(\pi)$ events are mutually independent, and
			\item[\tu{c}] For all $j$, we have $\Pr[A_j(\pi)]=\Pr[L_{m,k}(\sig)=k]$ where $\sig\in \S_{m,k}$ is chosen uniformly at random.  
		\end{itemize}
		\end{claim}
		\begin{proof}
		    Part (a) is clear, and (b) follows from the fact that the $A_j(\pi)$ events involve the relative ordering of disjoint sets of letters.   For (c), one can consider the map which sends $\pi\in \S_{m,n}$ to $\sig \in \S_{m,k}$ by deleting every letter in $\pi$ except for $(jk+1),\ldots,((j+1)k)$ and then relabeling $jk+i$ to $i$.  It is not difficult to see that $A_j(\pi)$ occurs if and only if $L_{m,k}(\sig)=k$ occurs, and that $\pi$ being chosen uniformly at random implies $\sig$ is chosen uniformly at random.
		\end{proof}
		Let $p_k=\Pr[L_{m,k}(\sig)=k]$ and let $t$ be the integer such that $t!\le n<(t+1)!$.  The claim above implies that for all $k$ we have
		{\small \begin{equation}\Pr[L_{m,n}(\pi)\ge k]\ge \Pr\l[\bigcup A_j(\pi)\r]=1-\Pr\l[\bigcap A_j^c(\pi)\r]=1-(1-p_k)^{\floor{n/k}}\ge 1- \exp\l(-\f{ t!p_k}{2k}\r),\label{eq:expLower}\end{equation}}%
		with this last step using $\floor{n/k}\ge n/2k$ for $k\le n$, that $1-p_k\le e^{-p_k}$, and that $n\ge t!$.
		
		It is easy to see by definition that $p_k\ge 1/k!$ for all $m,k$; and for $n$ sufficiently large, we have $-e^{-t/2}\ge -t^{-1}$.  For such $n$, by \eqref{eq:expLower} we have for all $k\le t-2$ that
		\begin{equation}\Pr[L_{m,n}(\pi)\ge k]\ge 1- \exp\l(-\f{t!}{2k\cdot k!}\r) \ge 1-\exp\l(-\f{t}{2}\r)\ge 1-t^{-1}.\label{eq:small}\end{equation}
		Summing this bound over all $k\le t-2$ for $m=1$ gives 
		\[\E[L_{1,n}(\pi)]\ge \sum_{k\le t-2} \Pr[L_{1,n}(\pi)\ge k]\ge t-3.\]
		This gives the desired lower bound of $\Gam^{-1}(n)+\Om(1)$ for $m=1$ since $t\ge \Gam^{-1}(n)-2$.
		
		We now consider $m\ge 2$.  By Proposition~\ref{prop:lowerCont} we have for $k$ sufficiently large in terms of $m$ that $p_k\ge \f{(m/1.03)^k}{2k\cdot k!}$.  Let $n$ be large enough in terms of $m$ so that this bound holds for $k\ge t$.  Also let $n$ be large enough so that $\f{\log m}{\log t}\le 1$.  By Lemma~\ref{lem:asy}, if $t\le k\le t+\f{\log m}{100\log t}t\le 1.01t$, then $k!\le t! (1.1)^k m^{t/100}$.  Thus by \eqref{eq:expLower} we have
		\begin{align*}\Pr[L_{m,n}(\pi)\ge k]&\ge 1-\exp\l(-\f{(m/1.03)^k}{4k^2 (1.1)^k m^{t/100}}\r)\ge 1-\exp\l(-\f{m^{.99t}}{4k^2\cdot (1.14)^{1.01t}}\r)\\&\ge 1-\exp\l(-\f{(1.7)^t}{8t^2}\r),\end{align*}
		where this last step used $m\ge 2$.  This quantity is at least $1/2$ for $n$ (and hence $t$) sufficiently large.  Using this together with \eqref{eq:small} for $k\le t-2$ gives, for $n$ sufficiently large in terms of $m$,
		 \[\E[L_{m,n}(\pi)]\ge t-3+\sum_{t\le k\le t+\f{\log m}{100\log t}t} \Pr[L_{m,n}(\pi)\ge k] \ge t-3+\l(\f{\log m}{100\log t}t\r)\cdot \half,\]
		 proving the desired result.
	\end{proof}
	%\SS{I think taking $k\le t+(1-100/\log m)\f{\log m}{\log t}t$ would let us prove an asymptotic lower bound of $t+(\log m/\log t)t+\Om(1+t/\log t)$, which I suppose isn't as meaningful if we think of $m$ as fixed.  In any case the result becomes uglier unless we can get the same error term as the other guy, which requires looking at $k\le t+(1-100/t)(\log m/\log t)t$ and this in turn requires better bounds on the 1.03 type bounds.}
	
	%\SS{In terms of dependencies: we need $t\ge m$ which basically means $n\ge e^m$.  More substantially we need $k\ge m^{\log m}$ when $k\ge t$, and this means we roughly need $n\ge e^{m^{\log m}}$.  Note that this is the worst dependency we have, so the entire proof goes through as long as this dependency holds.}
	
	\section{Concluding Remarks}
	In this paper we solved a conjecture of Diaconis, Graham, He, and Spiro~\cite{DGHS21} by asymptotically determining $\E[L^1_{m,n}(\pi)]$ provided $n$ is sufficiently large in terms of $m$.  Using similar ideas, it is possible to compute the asymptotic limit of the $r^{\text{th}}$ moment $\E[L^1_{m,n}(\pi)^r]$ for any fixed $r$.  Based off of computational evidence for these higher moments, we conjecture the following:

	%Computational evidence suggests that $L^1_{m,n}(\pi)$ converges to a normal distribution as $m$ tends to infinity provided $n$ is sufficiently large in terms of $m$.  One approach to proving this would be to show the following.
	\begin{conj}\label{conj:strong}
	For all $r\ge 1$, if $n$ is sufficiently large in terms of $m$, then
	\[\E[(L^1_{m,n}(\pi)-\mu)^r]=c_r m^{\floor{r/2}} +O(m^{\floor{r/2}-1}),\]
	where $\mu=\E[L^1_{m,n}(\pi)]$ and
	\[c_r=\begin{cases} \f{r!}{2^{r/2}(r/2)!} & r\tr{ even},\\
	\f{r!}{3\cdot 2^{(r-1)/2}((r-3)/2)!} & r\tr{ odd}.\end{cases}\]
	\end{conj}
	One can show that the standard deviation $\sig$ of $L^1_{m,n}(\pi)$ is asymptotic to $m^{1/2}$.  Thus, this conjecture would imply that the standardized moments $(\f{L^1_{m,n}(\pi)-\mu}{\sig})^r$ converge to 0 for $r$ odd and to $\f{r!}{2^{r/2}(r/2)!}$ for $r$ even.  These are exactly the moments of a standard normal distribution, and actually this fact would imply that $(L_{m,n}^1(\pi)-\mu)/\sig$ converges in distribution to a standardized normal distribution, see for example \cite[Corollary 21.8]{frieze2016introduction}.
	
	Perhaps a first step towards proving Conjecture~\ref{conj:strong} would be to get a better understanding of the (non-centralized) moments $\E[L_{m,n}^1(\pi)^r]$, and to this end we conjecture the following:
	\begin{conj}
	For all $r\ge 1$, if $n$ is sufficiently large in terms of $m$, then \[\E[L^1_{m,n}(\pi)^r]=m^r+{r+1\choose 2} m^{r-1}+O(m^{r-2}).\]
	\end{conj}
	We can prove that the $r^{\text{th}}$ moment is asymptotic to $m^r$, but we do not know how to determine the coefficient of $m^{r-1}$.  We were unable to observe any pattern for the coefficients of lower order terms.

	In this paper, we considered continuously increasing subsequences in multiset permutations, and it is natural to consider other types of subsequences in multiset permutations.  Perhaps the most natural to consider is the following:
	\begin{quest}\label{quest:longestInc}
	For $\pi\in \S_{m,n}$, let $\widetilde{L}_{m,n}(\pi)$ denote the length of a longest increasing subsequence in $\pi$.  What is $\E[\widetilde{L}_{m,n}(\pi)]$ asymptotic to when $m$ is fixed?
	\end{quest}
	When $m=1$ it is well known that $\E[\widetilde{L}_{1,n}(\pi)]\sim 2\sqrt{n}$ (see~\cite{Rom}), so Question~\ref{quest:longestInc} is a natural generalization of this classical problem.  See also recent work of Almeanazel and Johnson \cite{AMJ20} for some results concerning the distribution of $\widetilde{L}_{m,n}(\pi)$.
	
	\textbf{Acknowledgments}.  We thank S. Venkitesh for fruitful conversations and Persi Diaconis for comments regarding an earlier draft.  We thank the Graduate Student Combinatorics Conference 2021 for hosting an open problem session, at which the fourth author presented the problem which inspired this work.
	 \bibliographystyle{alpha}
	 \bibliography{references}

\end{document}